\documentclass[11pt]{article}
\usepackage{}

\usepackage{amsmath,amsthm,verbatim,amssymb,amsfonts,amscd,graphicx}

\usepackage{bm}
\usepackage{tikz}
\usetikzlibrary{arrows,shapes,positioning}
\usetikzlibrary{decorations.markings}
\tikzstyle arrowstyle=[scale=1]
\tikzstyle directed=[postaction={decorate,decoration={markings, mark=at position 0.75 with {\arrow[arrowstyle]{stealth}}}}]
\tikzstyle redirected=[postaction={decorate,decoration={markings, mark=at position 0.35 with {\arrow[arrowstyle]{stealth}}}}]

\newtheorem{theorem}{Theorem}[section]
\newtheorem{corollary}[theorem]{Corollary}
\newtheorem{definition}[theorem]{Definition}
\newtheorem{conjecture}[theorem]{Conjecture}

\newtheorem{lemma}[theorem]{Lemma}

\newtheorem{claim}{Claim}[theorem]
\newtheorem{proposition}[theorem]{Proposition}

\DeclareMathOperator{\supp}{{\rm supp}}

\DeclareMathOperator{\Z}{\mathbb{Z}}

\newcommand{\JCTB}{{\it J. Combin. Theory Ser. B}}

\newcommand{\JGT}{{\it J. Graph Theory}}

\newcommand{\DM}{{\it Discrete Math.}}

\newcommand{\SIAMDM}{{\it SIAM J. Discrete Math.}}

\newcommand{\CJM}{{\it Canad. J. Math.}}

\newcommand{\PLMS}{{\it Proc. London Math. Soc.}}

\newcommand{\EJC}{{\it European J. Combin.}}

\begin{document}
\title{Integer flows on  triangularly connected signed  graphs}
\author{Liangchen Li \thanks{Partially supported by NSFC (No. 12271438),  Basic Research Foundation of Henan Educational Committee (No. 20ZX004)}\\ School of Mathematical Sciences\\ Luoyang Normal University \\Luoyang  471934, China\\ Email: liangchen\_li@163.com. \\
Chong Li \\
Mathematics and Computer Science\\
 Southern Arkansas University\\
 100 E. University, Magnolia, Arkansas 71753-5000\\ Email: chongli@saumag.edu\\
Rong Luo\thanks{Partially supported by a grant from  Simons Foundation (No. 839830)}, Cun-Quan Zhang\thanks{Partially supported
 by an  NSF grant DMS-1700218} \\
Department of Mathematics\\
 West Virginia University\\
  Morgantown, WV 26505, USA\\
   Emails:   rluo@mail.wvu.edu, cqzhang@mail.wvu.edu
 }
\date{}
\maketitle

\vspace{-0.2cm}
\begin{abstract}
A triangle-path in a graph $G$ is a sequence of distinct triangles $T_1,T_2,\ldots,T_m$ in $G$ such that for any  $i, j$ with $1\leq i  < j  \leq m$, $|E(T_i)\cap E(T_{i+1})|=1$ and $E(T_i)\cap E(T_j)=\emptyset$ if $j > i+1$. A connected graph $G$ is triangularly connected if for any two nonparallel edges $e$ and $e'$ there is a triangle-path $T_1T_2\cdots T_m$ such that $e\in E(T_1)$ and $e'\in E(T_m)$.
For ordinary graphs, Fan {\it et al.}~(J. Combin. Theory Ser. B 98 (2008) 1325-1336) characterize all triangularly connected graphs that admit nowhere-zero $3$-flows or $4$-flows.
Corollaries of this result include
integer flow of some families of ordinary graphs, such as, locally connected graphs due to Lai (J. Graph Theory 42 (2003) 211-219) and some types of products of graphs due to Imrich et al.(J. Graph Theory 64 (2010) 267-276).
In this paper, Fan's result for triangularly connected graphs is
further extended to
 signed graphs.
We proved that a flow-admissible  triangularly connected signed graph  admits a nowhere-zero
$4$-flow if and only if  it is not
the wheel $W_5$ associated with a specific signature.
 Moreover, this result is sharp since there are infinitely many unbalanced   triangularly  connected signed graphs
admitting a nowhere-zero $4$-flow but no $3$-flow.

\end{abstract}
\vspace{-0.3cm}
\section{Introduction}
\vspace{-0.2cm}
Graphs  or signed graphs in this paper are finite and may have loops and parallel edges.  For terminology and notations not defined here we follow \cite{Bondy2008, Diestel2010, West1996}.

The theory of integer flows which is a dual problem to the vertex coloring of planar graphs was introduced by Tutte \cite{Tutte54,Tutte1949}.
Tutte's flow theory
has been extended to signed graphs \cite{Bouchet1983}.
 The concept of integer flows on signed graphs naturally comes from the study of graphs embedded on nonorientable surfaces, where nowhere-zero flow emerges as the dual of
 local tension.  Bouchet \cite{Bouchet1983}  in 1983 proposed the following conjecture.

\vspace{-0.2cm}
\begin{conjecture}
\label{CONJ: Bouchet}
{\rm (Bouchet \cite{Bouchet1983})}
Every flow-admissible signed graph admits a nowhere-zero $6$-flow.
\end{conjecture}

\vspace{-0.2cm}
The best result toward Conjecture~\ref{CONJ: Bouchet} is due to DeVos
{\it et al.}~\cite{DLLLZZ}  who prove that the conjecture is true with $6$ replaced by $11$.
 Integer flows on signed graphs also have been studied for many specific families of  graphs. For more details, the readers are refereed to   \cite {HuLi-wheel-2018, LLSSZ2018, LLZh-EJC2018, Kaiser2016, Macajova2016JGT, MS-eulerian, Raspaud2011, WLZZ, Wu2014}.

In this paper, we investigate  nowhere-zero  integer flows in  triangularly connected signed  graphs.
For  triangularly connected  ordinary graphs,
 Fan {\it et al.}~\cite{Fan2008} show that every  triangularly connected ordinary graph admits a nowhere-zero $4$-flow and they also characterize all such graphs not admitting a nowhere-zero $3$-flow.    For its signed counterpart, we prove the following result.

\begin{theorem}\label{Tri-5-flow}
If $(G,\sigma)$ is a flow-admissible  triangularly connected signed  graph, then $(G,\sigma)$ admits a nowhere-zero $4$-flow if and only if $(G,\sigma)\neq (W_5,\sigma^*)$ where $(W_5, \sigma^*)$ is the signed graph in Figure~\ref{FIG: 1}. Moreover there are infinitely many triangularly connected unbalanced signed graphs that admit a nowhere-zero $4$-flow but no $3$-flow.
\end{theorem}
\begin{figure}\begin{center}
\tikzstyle{help lines}+=[dashed]% aaarghhh!!!
\begin{tikzpicture}[scale=0.7]
%\draw[style=help lines] (-3,-3) grid +(6,6);
\draw [fill=black] (0,0) circle (0.06);
\draw [fill=black] (18:1.5) circle (0.06);
\draw [fill=black] (90:1.5) circle (0.06);
\draw [fill=black] (162:1.5) circle (0.06);
\draw [fill=black] (234:1.5) circle (0.06);
\draw [fill=black] (306:1.5) circle (0.06);
\draw[dashed] (18:1.5)--(90:1.5)--(162:1.5)--(234:1.5)--(306:1.5)--(18:1.5);
\draw (0,0)--(18:1.5) (0,0)--(90:1.5) (0,0)--(162:1.5) (0,0)--(234:1.5) (0,0)--(306:1.5);
\end{tikzpicture}
\caption{\small  $(W_5,\sigma^*)$ has a 5-NZF but no 4-NZF. Dotted edges are negative.}
\label{FIG: 1}
\end{center}
\end{figure}
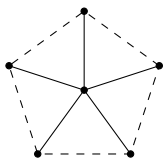

 A graph is {\it locally connected}  if the subgraph induced by the neighbor of each vertex is connected.
It is known that
locally connected graphs, square of graphs, chordal graphs, triangulations on surfaces, and some types of products of graphs are triangularly connected (such as \cite{Imrich2010},  \cite{Lai2003},
for ordinary graphs) and thus
  we have the following corollary.

\begin{corollary}
Let $(G,\sigma)$ be a flow-admissible signed graph. If $G$ is locally connected, then $(G,\sigma)$ admits a nowhere-zero $4$-flow if and only if $(G,\sigma)\neq (W_5,\sigma^*)$. In particular, if $G$ is the square of a connected graph or is the strong product of graphs, then $(G,\sigma)$ admits a nowhere-zero $4$-flow.
\end{corollary}

It is worth to pointing out that in the study of flows of ordinary graphs, Tutte's result  on the equivalence of modulo flows and integer flows serves as one of  most powerful tools (see \cite{Tutte1949}). However this equivalence no longer holds for signed graphs.  Therefore the approach  in the proof of our theorem is significantly different and difficult comparing with that of its ordinary counterpart.

The rest of the paper is organized as follows: Basic notations and definitions will be introduced in Section~\ref{se:defintions}.  Section~\ref{se:useful-lemma} will present some lemmas needed in the proofs of the main result.  In Section \ref{se:sharpness}, we will present a family of unbalanced triangularly connected  graphs  that admit  nowhere-zero $4$-flows but no $3$-flows.
The proof of Theorem \ref{Tri-5-flow} will be presented in Section~\ref{se:the proof}.

\vspace{-0.3cm}
\section{Notations and Terminology}
\label{se:defintions}
\vspace{-0.2cm}
A triangle-path of length $m$, denoted by $T_1T_2\cdots T_m$ in $G$ is a sequence of distinct triangles $T_1, T_2, \dots,  T_m$ in $G$ such that for  any $1\leq i < j \leq m$,

\centerline{$ |E(T_i)\cap E(T_{i+1})|=1 \mbox{  and } E(T_i)\cap E(T_j)=\emptyset \mbox{ if } j> i+1.$}

 A connected graph $G$ is triangularly connected if for any two nonparallel distinct edges  $e$ and $e'$, there is a triangle-path $T_1T_2\cdots T_m$ such that $e\in E(T_1)$ and $e'\in E(T_m)$. Trivially, the graph with a single edge is triangularly connected.     Let $H_1,H_2, \dots, H_t$ be  subgraphs of $G$. Denote by $H_1\triangle H_2 \triangle \cdots \triangle H_t$ the symmetric difference of those subgraphs.

Let $G$ be a graph. Let  $U_1$ and $U_2$ be two disjoint vertex sets. Denote by $\delta_G(U_1,U_2)$ the set of edges with one end in $U_1$ and the other in $U_2$. For convenience, we write $\delta_G(U_1)$ for $\delta_G(U_1,V(G)\setminus U_1)$. We use $B(G)$ to denote the set of bridges of $G$.  A path in $G$ is said to be a \emph{subdivided edge} of $G$ if every internal vertex of $P$  has degree $2$.

 A \emph{signed graph} $(G,\sigma)$ is a graph $G$ together with a {\em signature} $\sigma: E(G) \to \{-1,1\}$. An edge $e \in E(G)$ is \emph{positive} if $\sigma(e) =1$ and \emph{negative} otherwise. Denote the set of all negative edges of $(G,\sigma)$ by $E_N(G,\sigma)$ (or simply $E_N(G))$. For a vertex $v$ in $G$, we define a new signature $\sigma'$ by changing $\sigma'(e) = -\sigma(e)$ for each $e\in \delta_G(v)$.
We say that $\sigma'$ is obtained from $\sigma$ by making a \emph{switch} at the vertex $v$.
Two signatures are said to be {\em equivalent} if one can be obtained from the other by making a sequence of switch operations.

For convenience, the signature $\sigma$ is usually omitted if no confusion arises or is written as $\sigma_G$ if it needs to emphasize $G$. For a subgraph $H$ of $G$, denote by $(H,\sigma|_H)$ the signed graph where $\sigma|_H$ is the restriction of $\sigma$ on $E(H)$.
If there is no confusion from the context, we simply use $H$ to denote the signed subgraph.

Every edge of $G$ is composed of two half-edges $h$ and $\hat{h}$, each of which is incident with one end.
Denote the set of half-edges of $G$ by $H(G)$ and the set of half-edges incident with $v$ by $H_G(v)$.
For a half-edge $h \in H(G)$, we use  $e_{h}$ to refer to  the edge containing $h$.
An {\em orientation} of a signed graph $(G, \sigma)$ is a mapping $\tau: H(G) \to \{-1,1\}$
such that $\tau(h) \tau(\hat{h}) = -\sigma(e_{h})$ for each $h \in H(G)$.
It is convenient to consider $\tau$ as an assignment of orientations on $H(G)$.
Namely, if $\tau(h) =1$, $h$ is a half-edge oriented away from its end and otherwise towards its end. Such an ordered pair $(G,  \tau)$ is called a {\em bidirected graph}.

\begin{definition} Let $(G,\tau)$ be a bidirected graph,
 $A$ be an abelian group, and $f: E(G) \to A$ be a mapping. The pair $(\tau,f)$ (or to simplify, $f$) is an \emph{$A$-flow} of $G$ if $\sum_{h\in H_G(v)} \tau (h) f(e_h)=0$ for each $v\in V(G)$, and is an (integer) \emph{$k$-flow} if it is a $\mathbb{Z}$-flow and $|f(e)|<k$ for each $e\in E(G)$.
\end{definition}

Let $f$ be a flow of a signed graph $G$. The {\rm support} of $f$, denoted by $\supp (f)$, is the set of edges $e$ with $f(e)\neq 0$. The flow $f$ is {\em nowhere-zero} if $\supp (f) = E(G)$. For convenience, we abbreviate the notions of
{\em nowhere-zero $A$-flow} and
{\em nowhere-zero $k$-flow} as
{\em
$A$-NZF}
 and
 {\em $k$-NZF}, respectively.
 Observe that $G$ admits an $A$-NZF (resp., a $k$-NZF) under an orientation $\tau$ if and only if it admits an $A$-NZF (resp., a $k$-NZF) under any orientation $\tau'$.
A $\mathbb Z_k$-flow is also called a modulo $k$-flow. For an integer flow $f$ of $G$ and a positive integer $t$, let $E_{f=\pm t} =\{e\in E(G) : |f(e)|=t\}$. For any subgraph $H$ of $G$, denote $f(H) = \{f(e) : e\in E(H)\}$.

 A circuit is \emph{balanced} if it contains an even number of negative edges, and is \emph{unbalanced} otherwise. A signed graph is called {\it balanced} if it contains no unbalanced circuit and is called \emph{unbalanced} otherwise.   A balanced signed graph is equivalent to an ordinary graph. A signed circuit is defined as a signed graph of one of the following three types:

(1) a balanced circuit;

(2) a short barbell, the union of two unbalanced circuits that meet at a single vertex;

(3) a long barbell, the union of two vertex-disjoint unbalanced circuits with a path  with at least one edge that meets the circuits only at its ends.

\begin{proposition} {\rm (\cite{Bouchet1983})}
\label{signed circuit-flow}
Every balanced circuit or  short barbell has a $2$-NZF and every long barbell has a $3$-NZF where  an edge has flow value $2$ or $-2$ if and only if it belongs to the path connecting the two unbalanced circuits.
\end{proposition}

A signed graph $G$ is {\em flow-admissible} if it admits a $k$-NZF for some  positive integer $k$. Bouchet \cite{Bouchet1983} characterized all flow-admissible signed graphs as follows.

  \begin{proposition} {\rm (\cite{Bouchet1983})}
  \label{flow admissible}
 Let $(G,\sigma)$ be a connected signed graph.  The following three statements are equivalent:

 (1) $(G, \sigma)$ is flow-admissible;

(2)    $(G,\sigma)$  is not equivalent to a signed graph with exactly one negative edge and it has no cut-edge $b$ such that $(G-b,\sigma|_{G-b})$ has a balanced component;

 (3)  every edge in $(G,\sigma)$ is contained in a signed circuit.
 \end{proposition}

\section{Lemmas}
\label{se:useful-lemma}

In this section, we will present some lemmas that will be used  in the proof of our main result.

Let $H$ be a signed graph and $C$ be a balanced circuit. Define the following operation:

\centerline{$\Phi_2\mbox{-operation}: \mbox{ add a balanced circuit $C$ to $H$ if $|E(C)\setminus E(H)|\leq 2$.}$}

Let $H$ be a subgraph of $G$. We use $\langle H\rangle_2$ to denote the maximal subgraph of $G$ obtained from $H$ via $\Phi_2$-operations. Z\'{y}ka \cite{Zyka1987} proved the following result.

\begin{lemma}{ \rm (Z\'{y}ka \cite{Zyka1987})}\label{2-operation}
Let $(G,\sigma)$ be a signed graph and $H$ be a subgraph of $G$. If  $\langle H\rangle_2=(G, \sigma)$, then $(G,\sigma)$ admits a $\mathbb Z_3$-flow $\phi$ such that $E(G)\setminus E(H)\subseteq \supp(\phi)$.
\end{lemma}

\vspace{-0.3cm}

\begin{lemma} \label{2-operation-2}
Let $(G,\sigma)$ be a   triangularly connected signed graph. Let $T$ be an unbalanced triangle if there is one otherwise let $T$ be any balanced triangle. Then $\langle T\rangle_2=(G, \sigma)$ and  $(G,\sigma)$ has a $\Z_3$-flow $\phi$ such that $E_{\phi = 0} \subseteq E(T)$ and for any triangle $T'$, if there are two edges $e_1,e_2\in E(T')$ such that $T'$ is the only triangle containing them, then $\phi(e_1) = \phi(e_2)$.
\end{lemma}
\vspace{-0.3cm}
\begin{proof}  If there is a triangle $T'$ containing two  edges $uv,uw$ such that  each is contained in exactly one triangle which is $T'$, then split $u$ into two vertices $u_1$ and $u_2$ such that $u_1$ is adjacent to $v$ and $w$,  and $u_2$ is adjacent to each vertex in $N_G(u)-\{v,w\}$. Then the degree of $u_1$ is $2$. Repeating this operation until  every pair of such edges share a degree $2$-vertex. Denote the resulting graph by $(G', \sigma)$.

 It is clear that $\langle T\rangle_2=(G', \sigma)$. Thus by Lemma~\ref{2-operation}, $(G',\sigma)$ has a $\Z_3$-flow $\phi$ such that $E_{\phi = 0} \subseteq E(T)$.  Then $\phi$ is a desired $\Z_3$-flow of $(G,\sigma)$.
\end{proof}
\vspace{-0.3cm}

\begin{lemma}{\rm (Xu and Zhang \cite{Xu2005})} \label{eulerian-2-flow}
A signed graph $(G,\sigma)$ admits a $2$-NZF if and only if each component of $(G,\sigma)$ is eulerian and has an even number of negative edges.
\end{lemma}

The following two lemmas strengthen a result due to Xu and Zhang \cite{Xu2005}.

\begin{lemma}{\rm (DeVos et al. \cite{DLLLZZ})}\label{assign value}
Let $(G,\sigma)$ be a bridgeless signed graph admitting a $\mathbb Z_3$- NZF. Then for any edge $e'\in E(G)$ and for any $i\in \{1,2\}$, $(G,\sigma)$ admits a $3$-NZF $f$ such that $f(e')=i$.
\end{lemma}

The next lemma is proved in \cite{LLZZ20202}.  For the purpose of self-conaintment, we include their proof here.
\begin{lemma}
\label{Bridges:mod3-to-3}
Let $(G,\sigma)$ be a signed graph such that there is a path containing all the bridges. Then $(G,\sigma)$ admits a $3$-NZF if $(G,\sigma)$ admits a  $\mathbb Z_3$-NZF.
\end{lemma}
\vspace{-0.3cm}
\begin{proof}   Let  $(\tau,\phi)$ be a nowhere-zero $\mathbb Z_{3}$-flow of $(G,\sigma)$. We may assume $\phi(e) = 1$ for each edge $e$. By Lemma~\ref{eulerian-2-flow}, we may further assume that $G$ has bridges. Since there is a path containing all the bridges, $G$ has exactly two leaf blocks, say $G_1$ and $G_2$. Let $e_1= u_1v_1$ and $e_2 =u_2v_2$ be the  two bridges such that $u_i \in V(G_i)$ for each $i=1,2$.

For each $i=1,2$,  denote  by $G_i'$ the signed graph obtained from $G_i$ by adding a negative loop  $e_i'$ at $u_i$ such that the two half edges of $e_i'$ are oriented the same as the half edge of $e_i$ incident with $u_i$.   Then both $G_1$ and $G_2$ are bridgeless  and  each admits a $\mathbb Z_3$-NZF. By Lemma \ref{assign value},  $G_i'$ admits a  nowhere-zero $3$-flow  $g_i$ such that $g(e_i) =  1$  for each $i =1,2$

If $e_1$ and $e_2$ are distinct, denote by $G_3'$ the signed graph obtained by deleting $G_1$,  $G_2$, $e_1$ and $e_2$,  and then adding a new edge $e_3=v_1v_2$ where  $v_1v_2$ consists of the  half-edge of $e_1$ incident with $v_1$ with  the same orientation and the half-edge of  $e_2$ incident with $v_2$ with the same orientation.   Then $G_3$ is bridgeless and admits a $\mathbb Z_3$-NZF.  By Lemma \ref{assign value}, $G_3$ admits a nowhere-zero $3$-flow $g_3$ such that $g(e_3) =  2$.

If $e_1 = e_2$, then $e_1$ is the only bridge of $G$ and thus $G= G_1\cup G_2 \cup \{e_1\}$. It is easy to see that one can obtain a $3$-NZF  of $(G,\sigma)$ from $g_1$ and $g_2$ by   deleting the negative loops $e_1', e_2'$ and assigning $e_1$ with the flow value $2$, a contradiction

  If $e_1 \not = e_2$,
    one can merge $g_1,g_2$ and $g_3$ to  obtain a nowhere-zero $3$-flow  of $(G,\sigma)$, a contradiction. This completes the proof of the lemma.
\end{proof}

The following lemma directly follows  from the definition of triangularly connected graphs.
\begin{lemma}\label{3-edge-cut}
Let $G$ be a triangularly connected graph and   $U,W$ be two disjoint vertex set with  $|\delta_G(U,W)|=3$.
   Then either  the three edges in  $\delta_G(U,W)$ share a common end vertex or the three edges induce a path on four vertices. Moreover in the latter case, the four vertices of the path induce a $K_4$ minus one edge.
\end{lemma}

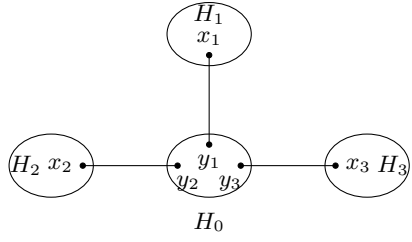
\begin{figure}\begin{center}
\tikzstyle{help lines}+=[dashed]% aaarghhh!!!
\begin{tikzpicture}[scale=0.7]
%\draw[style=help lines] (-3,-3) grid +(6,6);
\draw [fill=black] (0,0.4) circle (0.05);
\draw [fill=black] (0,2.1) circle (0.05);
\draw [fill=black] (0.6,0) circle (0.05);
\draw [fill=black] (-0.6,0) circle (0.05);
\draw [fill=black] (2.4,0) circle (0.05);
\draw [fill=black] (-2.4,0) circle (0.05);
\draw (0,0) ellipse (0.8 and 0.6 );
\draw (3,0) ellipse (0.8 and 0.6 );
\draw (-3,0) ellipse (0.8 and 0.6 );
\draw (0,2.5) ellipse (0.8 and 0.6 );
\draw (0,0.4)--(0,2.1) (0.6,0)--(2.4,0) (-0.6,0)--(-2.4,0);
\node[below] at (0,0.4){\footnotesize $y_1$};\node[above] at (0,2.1){\footnotesize $x_1$};
\node[below] at (-0.4,0){\footnotesize $y_2$};\node[below] at (0.4,0){\footnotesize $y_3$};
\node[left] at (-2.4,0){\footnotesize $x_2$};\node[right] at (2.4,0){\footnotesize $x_3$};
\node[left] at (-3,0){\footnotesize $H_2$};\node[right] at (3,0){\footnotesize $H_3$};
\node[above] at (0,2.5){\footnotesize $H_1$};\node[below] at (0,-0.7){\footnotesize $H_0$};
\end{tikzpicture}
\caption{\small  The structure of a graph with three bridges not contained in a path}
\label{FIG: G'}
\end{center}
\end{figure}
\vspace{-0.3cm}
\begin{lemma}
\label{LE:4-edges}
Let $G$ be a triangularly connected graph with $\delta(G)\geq 3$ and $E_0$ be a set of edges of $G$. If $|E_{0}|\leq 4$ and  each component of $G-E_0$ is either an isolated vertex or has minimum degree at least 2, then  in each nontrivial component, there is a path containing all the bridges of the component.
\end{lemma}
\vspace{-0.3cm}
\begin{proof}
Suppose to the contrary that $G'=G-E_{0}$  has a component say $H$ that contains three bridges, say $x_1y_1,x_2y_2,x_3y_3$, which don't belong to a path (see Figure \ref{FIG: G'}). Deleting these three edges, we will get four components and denote the component containing $x_i$ by $H_i$ for $i=1,2,3$ and denote the component containing $y_1,y_2,y_3$ by $H_0$.

Since $G$ is triangularly connected and $\delta(G) \geq 3$,  $G$ has no cut-vertex and has no $2$-edge-cut.  Thus $G$ is $3$-edge connected.
Since the minimum degree of each nontrivial component of $G-E_0$ is at least $2$, $|V(H_i)|\geq 2$ for each $i=1,2,3$.

\begin{claim}
$G'$ is connected.
\end{claim}
\noindent\emph{Proof.} Suppose to the contrary that $G_1$ and $G_2$ are two components of $G'$, where $H_i\subseteq G_1$ for each $i=1,2,3$. This implies that $2|E_{0}|=|\delta_G(G_2)|+\sum_{i=0}^3|\delta_G(V(H_i))|-3\times 2\geq 9$. It contradicts  the hypothesis  $|E_0|\leq 4$. $\square$

\begin{claim}
\label{H-1}
There exists an integer $i\in \{1,2,3\}$  such that $\delta_G(H_i)=3$ and for any $H_j$ with $|\delta_G(H_j)| =3$,  $\delta_G(H_j)\neq\delta_G(H_j,H_0).$
\end{claim}
\noindent\emph{Proof.} We first prove that there exists an $i\in \{1,2,3\}$  such that $\delta_G(H_i)=3$. Suppose to the contrary that $\delta(H_j)\geq 4$ for each $j\in \{1,2,3\}$. It follows that $2|E_{0}|=\sum_{j=1}^3|\delta_G(H_j)|+|\delta_G (H_0)|-3\times 2\geq  3\times 4 + 3 -6=9$, a contradiction.

Without loss of generality,  assume that $|\delta_G(H_1)|=3$.   Suppose to the contrary that $\delta_G(H_1)=\delta_G(H_1,H_0)$. It follows that $|\delta_G(H_2)|=|\delta_G(H_3)|=3$, otherwise $2|E_{0}|=\sum_{j=1}^3|\delta_G(H_j)|+|\delta_G (H_0)|-3\times 2\geq  9$, a contradiction. If $|\delta_G(H_2)\cap \delta_G(H_3)|\leq 1$, then $|E_{0}|\geq \sum_{j=1}^3(|\delta_G(H_j)|-1)-1\geq 5$, a contradiction. Thus $|\delta_G(H_2)\cap \delta_G(H_3)|=2$. This implies that $\{x_2y_2,x_3y_3\}$ is a 2-edge-cut of $G$. It contradicts that $G$ is 3-edge-connected.  Therefore  $\delta_G(H_1) \not =\delta_G(H_1,H_0)$. This completes the proof of this claim.
$\square$

By Claim~\ref{H-1},  in the following without loss of generality we assume that $|\delta_G(H_1)|=3$ and  $\delta_G(H_1,H_2) \not = \emptyset$.

\begin{claim}
\label{H1}
$\delta_G(H_1,H_3) = \emptyset$, $|\delta_G(H_3)|=3$, and $\delta_G(H_2,H_0)=\{x_2y_2\}$.
\end{claim}
\noindent\emph{Proof.}   Suppose to the contrary that   $\delta_G(H_1,H_3) \not = \emptyset$.  Since $\delta_G(H_1,H_2)  \not= \emptyset$, by  Claim~\ref{H-1}, we have  $|\delta_G(H_1,H_i)| = 1$ for each $i  =0, 2, 3$.  Since $\delta_G(H_1)$ is an edge cut with  $|\delta_G(H_1)|=3$ and clearly the three edges in $\delta_G(H_1)$ don't induce a path,  by Lemma~\ref{3-edge-cut},  the three edges share a common end vertex which is $x_1$. Since $|V(H_1)| \geq 2$,  we have that $x_1$ is a cut-vertex, a contradiction. This proves  $\delta_G(H_1,H_3) = \emptyset$.

Since $\delta_G(H_1,H_3) = \emptyset$ and $|E_0| \leq 4$, we have $3 \leq |\delta_G(H_3)| \leq 4 -2 +1=3$. Thus $|\delta_G(H_3)|= 3$.

Since  $(\delta_G(H_1)\cup \delta_G(H_3))\setminus \{x_1y_1, x_3y_3\} \subseteq E_0$ and $|(\delta_G(H_1)\cup \delta_G(H_3))\setminus \{x_1y_1, x_3y_3\} | = 4$, we have  $(\delta_G(H_1)\cup \delta_G(H_3))\setminus \{x_1y_1, x_3y_3\} = E_0$. Therefore $\delta_G(H_2,H_0) = \{x_2y_2\}$.
$\square$

\medskip \noindent
{\bf The final step.}  By Claims~\ref{H-1} and \ref{H1},  there is an edge $u_1u_2 \in \delta_G(H_1,H_2)$ where $u_1\in V(H_1)$ and $u_1 \not = x_1$. By Lemma~\ref{3-edge-cut}, $u_2$ and $y_1$ are adjacent. Since $\delta_G(H_2,H_0)=\{x_2y_2\}$ by Claim \ref{H1},  we have $u_2 = x_2$ and $y_1=y_2$.  Similarly there is an edge $v_3v_2 \in \delta_G(H_3,H_2)$ where $v_3\in V(H_3)$ and $v_3 \not = x_3$  and $v_2 = x_2$.   By Lemma \ref{3-edge-cut},  all the edges in $\delta_G(H_2)$ share a common end vertex $x_2$. Since $|V(H_2)|\geq 2$, $x_2$ is a cut-vertex, a contradiction to the fact that $G$  has no cut-vertex. This contradiction completes the proof of the lemma.
\end{proof}

The following is a corollary of Lemmas~\ref{Bridges:mod3-to-3} and \ref{LE:4-edges}.
\begin{lemma}\label{4-edges:mod3-to-3}
Let $(G,\sigma)$ be a  triangularly connected signed graph and $\phi$ be a $\mathbb Z_3$-flow of $(G,\sigma)$ with $|E_{\phi=0}|\leq 4$, then $(G,\sigma)$ admits a $3$-flow $f$ with $\supp(f)=\supp(\phi)$.
\end{lemma}

\vspace{-0.3cm}

\begin{lemma}\label{f-extension}
Let $k\geq 3$ be an integer and $C$ be a balanced circuit of  $(G,\sigma)$. Let $g$ be a $2$-flow of $(G,\sigma)$ with $\supp(g) = E(C)$ and  $f_1$ be  an integer $k$-flow of $(G, \sigma)$ such that  $|\supp(f_1) \cap E(C)|\leq k-2$ and $|f_1(e)| \leq \frac{k}{2}$ for each $e\in E(C)$. Then there is an $\alpha \in \{\pm 1, \pm 2, \cdots,\pm \lfloor\frac{k}{2}\rfloor\}$  such that $f_2 = f_1- \alpha g$ is  an integer $k$-flow with $\supp(f_2) = \supp(f_1)\cup E(C)$.
\end{lemma}
\vspace{-0.3cm}
\begin{proof}  Since $|\supp(f_1) \cap E(C)|\leq k-2$, we have $|f_1(C)| \leq k-1$.

 If $k$ is odd,  then there exists an  integer $\alpha\in\{\pm 1,
 \dots, \pm \lfloor\frac{k}{2}\rfloor\}\setminus f_1(C)$.

  If $k$ is even, then there exists at least two integers in  $\{\pm 1, \dots, \pm\frac{k}{2}\}\setminus f_1(C)$. If $\{\pm \frac{k}{2}\}\cap f_1(C)=\emptyset$, let $\alpha=\frac{k}{2}$; otherwise  pick one $\alpha\in\{\pm 1, \cdots, \pm (\frac{k}{2} - 1)\}\setminus f_1(C)$.   Let $f_2 = f_1-\alpha g$.
  
  Clearly, when $|\alpha| < \frac{k}{2}$,  $f_2$ is an integer $k$-flow with $\supp(f_2) = \supp(f_1)\cup E(C)$.
  
  If $\alpha = \frac{k}{2}$,  then $\{\pm \frac{k}{2}\}\cap f_1(C)=\emptyset$. Thus  for each $e\in E(C)$, $|f_1(e)|\leq \frac{k}{2} -1$, so $-(k-1) \leq f_2(e) = f_1(e) -\alpha g(e) \leq k-1$ and $ f_2(e) \not = 0$. Therefore, $f_2$ is an integer $k$-flow with $\supp(f_2) = \supp(f_1)\cup E(C)$. This completes the proof of the lemma.
    \end{proof}

\vspace{-0.3cm}

\begin{lemma}\label{Z3 0:1-2}
Let $C$ be a balanced circuit of $(G,\sigma)$ with length at most $4$ and $g$ be a $2$-flow of $(G,\sigma)$ with $\supp(g_1) = E(C)$. Then for any  $\mathbb Z_3$-flow $\phi$ of  $(G,\sigma)$,  there is an $\alpha \in \mathbb Z_3$ such that $\phi_1= \phi -\alpha g$ is a $\mathbb Z_3$-flow satisfying $|E_{\phi_1  =0}\cap E(C)| \in \{0, |E(C)| -2\}$.
\end{lemma}
\vspace{-0.3cm}
\begin{proof}
Let $\phi$ be a $\mathbb Z_3$-flow of  $(G,\sigma)$.  If $|E_{\phi  =0}\cap E(C)| \in \{0, |E(C)| -2\}$, take $\alpha = 0$.

 If $|E_{\phi  =0}\cap E(C)| \ge |E(C)| -1$,  we can easily find some $\alpha \in \mathbb Z_3$ such that $\phi_1= \phi-\alpha g_1$ is a $\mathbb Z_3$-flow satisfying  $|E_{\phi_1  =0}\cap E(C)|=0$.

 Now we assume $|E_{\phi  =0}\cap E(C)| \leq |E(C)| -3$ and $|E_{\phi_1  =0}\cap E(C)| \not \in \{0, |E(C)| -2\}$. Then  $|E(C)| = 4$ and $|E_{\phi  =0}\cap E(C)|=|E(C)|-3=1$.    Thus  $|\phi(C)| \in \{2, 3\}$. If $|\phi(C)| = 2$, then choose an $\alpha$ in   $\mathbb Z_3\setminus \phi(C)$.  If $|\phi(C)| =3$, then there is an $\alpha \in \phi(C) \setminus \{0\}$ such that there are exactly two edges $e$ in $E(C)$ with $\phi(e) = \alpha$. Then $\phi_1= \phi- \alpha g$ is a $\mathbb Z_3$-flow satisfying $\phi(e) = \phi_1(e)$ for each $e \in E(G)-E(C)$ and $|E_{\phi_1  =0}\cap E(C)| \in \{0, |E(C)| -2\}$.
\end{proof}

\vspace{-0.3cm}
\begin{lemma}
\label{E0:2-circuits}
Let $(G, \sigma)$ be a  triangularly connected signed graph and $C_1,\dots,C_t$  $(1 \leq t \leq 2)$ be pairwise edge-disjoint balanced circuits of length at most $4$. If $\phi$ is a $\mathbb Z_3$-flow of  $(G,\sigma)$ such that $E_{\phi=0}\subseteq \cup_{i=1}^t E(C_i)$, then $(G,\sigma)$ admits a $4$-NZF.
\end{lemma}
\vspace{-0.3cm}
\begin{proof}
By Lemma~\ref{Z3 0:1-2},  we may assume that  $|E_{\phi=0}\cap E(C_i)|\in\{0, |E(C_i)|-2\}$ for each $i=1,\dots, t$. Then $|E_{\phi=0}| \leq 4$. By Lemma \ref{4-edges:mod3-to-3}, there is a $3$-flow $f$ such that $\supp(f) = \supp(\phi)$ and of course $f$ is a $4$-flow.   Taking $k= 4$, we have $|f(e)| \leq \frac{k}{2}$ and  $|E_{f \not =0}\cap E(C_i)| = 2= k-2$ for each $C_i$ with $E_{f=0}\cap E(C_i) \not = \emptyset$. Applying Lemma~\ref{f-extension} on each $C_i$ with $E_{f=0}\cap E(C_i) \not = \emptyset$, one can obtain  a desired $4$-NZF.
\end{proof}

By Lemma 2.2 of \cite{Fan2008}, the proof of the following lemma is  straightforward.

\begin{lemma}
\label{shift-0}
Let $f$ be a $\mathbb Z_3$-flow of  $(G,\sigma)$ and $H=T_1T_2\cdots T_m$ be a triangle-path in $G$ such that each $T_i$ is balanced for $1\leq i\leq m$. Given an edge $e_0\in E(H)$, then there is another
$\mathbb Z_3$-flow $g$ of $(G,\sigma)$ satisfying:

(1)  $f(e) = g(e)$ for each $e\not \in E(H)$;

(2)  $g(e) \not = 0$ for each edge $e\in E(H) -\{e_0\}$.
\end{lemma}

\begin{lemma}\label{E0:34-circuits}
Let $(G, \sigma)$ be a  triangularly connected signed graph,  $C_1$ be a balanced triangle and  $C_2$ be a balanced circuit of length at most $4$ such that $|E(C_1)\cap E(C_2)|\leq 1$.  If $\phi$ is a $\mathbb Z_3$-flow  of $G$ such that $E_{\phi=0}\subseteq  E(C_1)\cup E(C_2)$,  then $(G,\sigma)$ admits a $4$-NZF.
\end{lemma}
\begin{proof}
If $C_1$ and $C_2$ are edge-disjoint, then  by Lemma~\ref{E0:2-circuits}, $(G,\sigma)$ admits a $4$-NZF.

 If $C_1$ and $C_2$ are not edge-disjoint, then $|E(C_1)\cap E(C_2)|= 1$. Let $e_0$ be the common edge of $C_1$ and $C_2$. Applying Lemma~\ref{shift-0} on $H= C_1$ and $e_0$, we may assume   $E_{\phi=0}\subseteq   E(C_2)$.  By Lemma~\ref{E0:2-circuits}, $(G,\sigma)$ admits a $4$-NZF.
\end{proof}
\vspace{-0.3cm}

\section{Sharpness of Theorem  \ref{Tri-5-flow}}
\label{se:sharpness}
 Fan et al. \cite{Fan2008} give a complete characterization of  triangularly connected ordinary graphs that admit a  $4$-NZF but no $3$-NZF. In this subsection we present a family of unbalanced  triangularly connected signed graphs that  admit a  $4$-NZF but no $3$-NZF.  Interestingly all those graphs do not contain an unbalanced triangle. This indicates that there are  unbalanced triangularly connected signed  graphs without unbalanced triangles.

 For each integer $t \geq 4$,  construct the signed graph $(G_{2t}, \sigma)$ as follows (see Figure \ref{FIG: unblanced} for an illustration with $t=4$):

 (1)  The  graph $G_{2t}$  is constructed from the two circuits $C_1=x_1x_2\cdots x_tx_1$ and $C_2=y_1y_2\cdots y_ty_1$ by adding the edges $y_ix_i$ and $y_ix_{i+1}$ for each $i \in Z_t$;

 (2)  $E_N(G_{2t},\sigma)$ consists of the edges $x_1x_2,y_1y_2$ and all edges $y_ix_i,y_ix_{i+1}$ except $y_1x_2$.

  \begin{figure}\begin{center}
\tikzstyle{help lines}+=[dashed]% aaarghhh!!!
\begin{tikzpicture}[scale=0.6]
%\draw[style=help lines] (-3,-3) grid +(6,6);
%\draw [fill=black] (0,0) circle (0.06);
\draw [fill=black] (0,3) circle (0.06);
\draw [fill=black] (-3,0) circle (0.06);
\draw [fill=black] (0,-3) circle (0.06);
\draw [fill=black] (3,0)circle (0.06);

\draw [fill=black] (-0.75,0.75) circle (0.06);
\draw [fill=black] (-0.75,-0.75) circle (0.06);
\draw [fill=black] (0.75,-0.75) circle (0.06);
\draw [fill=black] (0.75,0.75) circle (0.06);
\draw (0,3)--(-3,0)--(0,-3)--(3,0);
\draw[dashed] (3,0)--(0,3);
\draw (-0.75,0.75)--(-0.75,-0.75)--(0.75,-0.75)--(0.75,0.75);
\draw[dashed] (0.75,0.75)--(-0.75,0.75);
\draw[dashed] (0,3)--(-0.75,0.75)--(-3,0)--(-0.75,-0.75)
--(0,-3)--(0.75,-0.75)--(3,0)--(0.75,0.75);
\draw (0,3)--(0.75,0.75);
\node[below] at (-1,1.3){$x_1$};\node[above] at (0,3){$y_1$};
\node[below] at (-1,0){$x_4$};\node[below] at (1,-0.1){$x_3$};
\node[below] at (1,1.3){$x_2$};\node[right] at (3,0){$y_2$};
\node[left] at (-3,0){$y_4$};\node[below] at (0,-3){$y_3$};

\end{tikzpicture}
\caption{\small  an unbalanced signed graph $(G_8,\sigma)$}
\label{FIG: unblanced}
\end{center}
\end{figure}
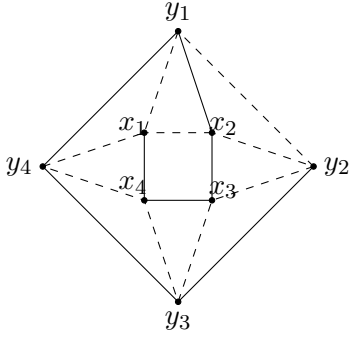

\begin{theorem}
\label{sharp}
For each $t\ge 4$, $(G_{2t},\sigma)$ is flow-admissible and  admits a $4$-NZF but no $3$-NZF.
\end{theorem}

Since $(G_{2t},\sigma)$ is bridgeless and every edge is contained in a balanced triangle,  by Proposition~\ref{flow admissible}, it is flow-admissible.
 Since $G_{2t}$ is  Eulerian,   the second part of Theorem~\ref{sharp} follows from the following  result due to Ma\v{c}ajova  and \v{S}koviera.

\begin{theorem}(Ma\v{c}ajova  and \v{S}koviera\cite{MS-eulerian})
\label{eulerian}
Let $(G, \sigma)$ be an   Eulerian signed graph with an odd number of negative edges.  Then $(G,\sigma)$ admits a  $4$-NZF if it is flow-admissible. Moreover $(G,\sigma)$ admits a  $3$-NZF if and only if $(G,\sigma)$ can be decomposed into three  signed Eulerian subgraphs that have a vertex in common and that each has an odd number of negative edges.
\end{theorem}

\vspace{-0.3cm}

\section{Proof of Theorem \ref{Tri-5-flow}}
\label{se:the proof}
We prove Theorem \ref{Tri-5-flow} by contradiction. Let $(G,\sigma)$ be a  counterexample such that $\beta(G) = \sum_{v\in V(G)} (d(v) -2)$ is as small as possible. Let $\tau$ be a fixed orientation of $(G,\sigma)$ in the proof.

Hu and Li \cite{HuLi-wheel-2018} show that $(W_5,\sigma^*)$ in Figure \ref{FIG: 1} admits a   $5$-NZF but no $4$-NZF.    Then $(G,\sigma)$ does not admit a $4$-NZF.   By Lemma \ref{E0:2-circuits} we have the following fact which will be applied frequently in the proof.

\medskip \noindent
{\bf Fact A}  $(G,\sigma)$ does not admit a $\Z_3$-flow $\phi$ such that $E_{\phi=0} \subseteq  E(C_1)\cup \cdots \cup E(C_t)$ where  $1 \leq t \leq 2$ and  $C_1,  \dots, C_t$ are   edge-disjoint balanced circuits of length at most four.

\medskip
 If $G$ contains two parallel edges $e_1$ and $e_2$, then  after inserting a degree $2$-vertex into $e_1$, the resulting graph $G'$  remains triangularly connected,  flow-admissible, and $\beta(G') = \beta(G)$. Thus  in the following proof, we assume  that $G$ is simple.

If $G$ contains no unbalanced triangle, let $T$ be a triangle. By Lemma~\ref{2-operation-2}, let $\phi$ be  a $\mathbb Z_3$-flow $\phi$ with $E_{\phi = 0} \subseteq E(T)$, a contradiction to Fact A. Thus  $G$ contains an unbalanced triangle.

\medskip \noindent
{\bf (I)} {\em $(G,\sigma)$ contains  two   edge-disjoint unbalanced triangles.}

\medskip \noindent
{\bf Proof of (I).}  Suppose to the contrary that $(G,\sigma)$ contains no edge-disjoint unbalanced triangles.  Let $T$ be an unbalanced triangle and $\phi$  be a $\mathbb Z_3$-flow $\phi$ with $E_{\phi = 0} \subseteq E(T)$.

 We consider two cases in the following.

\medskip \noindent
{\bf Case I.1.} $(G,\sigma)$ contains at least two unbalanced triangles.

Let $T_1, T_2, \dots, T_t$ be all the unbalanced triangles where $T=T_1$. Then $t \geq 2$. Since $(G,\sigma)$ contains no edge-disjoint unbalanced triangles, all unbalanced triangles share a common edge, denoted  by $uv$. For each $i$ denote by $w_i$ the third vertex of $T_i$.   Then for any $1 \leq i < j \leq t$,  $T_i\bigtriangleup T_j$ is a balanced circuit of length $4$.

Since $T_1\bigtriangleup T_2$ is a balanced $4$-circuit, by Fact A, $\phi(uv) = 0$ and $uv$ is not contained in a  balanced triangle.  This implies that no other triangle than $T_1, T_2,\dots, T_t$ contains $uv$.

Since $(G,\sigma)$ is flow-admissible, there is a signed circuit  $C$ containing $uv$.  By Proposition~\ref{signed circuit-flow}, let $f$ be a $2$-flow (if $C$ is a balanced circuit or a short barbell)  or a $3$-flow (if $C$ is a long barbell) such that $\supp(f) = E(C)$.  Let $\phi_1 = \phi + f$ be the $\mathbb Z_3$-flow. Then $\phi_1(uv) \not = 0$.

Let $e\in E_{\phi_1 = 0}- \bigcup_{i=1}^t E(T_i)$. Then there is a triangle-path $S_1S_2\cdots S_k$ where $e\in S_k$, $uv \in S_1 \in \{T_1,T_2,\dots, T_t\}$, and $S_2,S_3,\dots, S_k$ are balanced.   Let $H=S_2S_3\cdots S_k$ and $e' = E(S_1)\cap E(S_2)$. By Lemma~\ref{shift-0},  there is a
$\mathbb Z_3$-flow $g$ of $(G,\sigma)$ satisfying:

(1)  $\phi_1(e) = g(e)$ for each $e\not \in E(H)$;

(2)  $g(e) \not = 0$ for each edge $e\in E(H) -\{e'\}$.

By applying the above operation on each edge in $E_{\phi_1 = 0}- \bigcup_{i=1}^t E(T_i)$, one can obtain a $\mathbb Z_3$-flow $\phi_2$ such that $E_{\phi_2=0} \subseteq \bigcup_{i=1}^t E(T_i) - \{uv\}$.

Denote $C_i = T_1\bigtriangleup T_i$ for each $i= 2,\dots, t$. Then each  $C_i$ is a balanced $4$-circuit.   For each $i = 2, \dots, t$, let  $f_i$ be a $2$-flow of $(G,\sigma)$ with $\supp(f_i) = E(C_i)$ and let $\alpha_i \in \mathbb Z_3- \{\phi_2 (uw_i)f_i(uw_i), \phi_2(vw_i)f_i(vw_i)\}$. Let $\phi_3 = \phi_2 - \sum_{i=2}^t \alpha_i f_i$. Then $\phi_3$ is a $\mathbb Z_3$-flow such that $E_{\phi_3= 0} \subseteq \{uw_1,vw_1\} \subseteq E(C_2)$, a contradiction to Fact A.

\medskip \noindent
{\bf Case I.2.} $(G,\sigma)$ contains only one unbalanced triangle.

Denote   $E(T) = \{e_1,e_2,e_3\}$.   If every edge in $E_{\phi=0}$ is  contained in a triangle other than $T$,  then every edge in $E_{\phi=0}$ is contained in a balanced triangle since $T$ is the only unbalanced triangle in $(G, \sigma)$. By Lemma~\ref{E0:2-circuits}, $|E_{\phi=0}| \geq 2$ and those balanced triangles are not edge-disjoint. This implies that there is a $K_4$ containing $T$ where $T$ is the only unbalanced triangle in the $K_4$. However, $T$ is the symmetric difference of the other three balanced triangles in the $K_4$. Thus $T$ is balanced, a contradiction. Therefore there is one edge in $E_{\phi=0}$  that is contained in only one triangle which is $T$.

 Since $(G,\sigma)$ is flow-admissible, there is another edge in $E(T)$ which is contained in a balanced triangle.  Without loss generality, assume that $e_1$ is contained in only one triangle, $\phi(e_1) = 0$ and $e_3$ is contained a balanced triangle. Note that by Lemma~\ref{2-operation-2}, if $e_2$ is not contained in a balanced triangle, then $\phi(e_1) = \phi(e_2) = 0$.

   Since $(G,\sigma)$ is flow-admissible, by Proposition \ref{flow admissible},  there is a signed circuit $C_1$ containing $e_1$ and there is a signed circuit $C_2$ containing  $e_2$.   We choose $C_2= C_1$ if there is a signed circuit containing both $e_1$ and $e_2$; otherwise choose any signed circuit  $C_2$ containing $e_2$.

   By Lemma~\ref{signed circuit-flow}, let $f_i$  be a $2$-flow or $3$-flow  of $(G,\sigma)$ with $\supp(f_i) = E(C_i)$ for each $i=1,2$.

    We construct another $\Z_3$-flow $\phi_1$  of $(G,\sigma)$  as follows:

Let $\alpha \in \Z_3 -\{0, \phi(e_2)f_2(e_2)\}$.  If $C_1 = C_2$,  then $f_1 = f_2$ and let $\phi_1 = \phi  - \alpha f_1$; if $C_1 \not = C_2$, then $f_1(e_2) = f_2(e_1) = 0$ and let $\phi_1 = \phi- \alpha (f_1 + f_2)$.

Then $E_{\phi_1=0} \cap \{e_1,e_2\} = \emptyset$ and every edge in $E_{\phi_1=0}$ is contained in a balanced triangle.  Similar to the argument in Case I.1, there is a $\mathbb Z_3$-flow $\phi_2$ such that $E_{\phi_2 = 0} \subseteq \{e_3\}$  if $e_2$ is not contained in a balanced triangle or $E_{\phi_2 = 0} \subseteq \{e_2, e_3\}$ otherwise, a contradiction to Fact A.

We obtain a contradiction  in either case and thus completes the proof of (I).
\hfill $\Box$

\medskip \noindent
{\bf (II)} {\em $G$ is locally connected.}

\medskip \noindent
{\bf Proof of (II).}  Suppose to the contrary that $G$ is not locally connected. Then there is a vertex $v\in V(G)$ such that $G[N_G(v)]$ is not connected. Since $G$ is triangularly connected, each component of $G[N_G(v)]$ is nontrivial. Let $H$ be a component of $G[N_G(v)]$. Split $v$ into two nonadjacent vertices $v'$ and $v''$ where $v'$ is adjacent to all vertices in $H$ and $v''$ is adjacent to all vertices in $N_G(v)-V(H)$.  The signs of all edges remain the same.  Denote the resulting signed graph by $(G', \sigma)$. By (I), $(G', \sigma)$ contains two edge-disjoint unbalanced triangles. Since $G'$ is connected and bridgeless, by Proposition~\ref{flow admissible},  $(G', \sigma)$ is flow-admissible. Obviously $\beta(G')<\beta(G)$ and $G'$ remains triangularly connected.   By the minimality of $\beta(G)$,  $(G',\sigma)$ admits a  4-NZF $f$. Identifying $v'$ and $v''$, one can easily obtain a 4-NZF of $(G,\sigma)$, a contradiction.  Therefore $G$ is locally connected. \hfill $\Box$

\medskip \noindent
{\bf (III)} {\em
 $(G,\sigma)$  does not  contain any of the $11$ configurations in  Figure \ref{FIG: confi}.}

 \medskip \noindent
{\bf Proof of (III).}
For a balanced circuit  or a short barbell $C$, denote  by $\chi(C)$ a $2$-flow of $(G,\sigma)$ with $\supp(\chi(C)) = E(C)$ guaranteed by Lemma~\ref{signed circuit-flow}.  In the following argument, all cases only involve  one $\chi(C)$  except one which involves  three balanced circuits with one common edge. Thus without loss of generality, we assume that  $\chi(C)$ is a nonnegative $2$-flow.

\medskip
Take $T= T_1$ if $(G,\sigma)$ contains $FC_i$ if  $i \in \{1,2,3,9,10\}$, $T=T_2$ if $(G,\sigma)$ contains $FC_{4}$ or $FC_{11}$, and  $T= T_3$ if $(G,\sigma)$ contains $FC_i$ if $i \in \{5,6,7,8\}$.

  Since in $FC_1$ or $FC_2$, $E(T_1)$ is contained in two edge-disjoint balanced circuits of length at most $4$, a contradiction to Fact A.  This proves that $(G,\sigma)$ does not contain $FC_{1}$ or $FC_2$.

  In $FC_3$,  any two edges in $T_1$ are contained in a balanced $4$-circuit, thus by Fact A, $E_{\phi = 0} = E(T_1)$. Let $C= T_2\bigtriangleup T_3$. Then $C$ is a  balanced $4$-circuit and  contains the  two edges $uv_1$ and $uv_2$.
   Let $\phi_1 = \phi + \phi(v_2v_3) \chi(C)$.
  Then $\phi_1$ is a $\mathbb Z_3$-flow such that $E_{\phi_1=0} \subseteq E(T_1\bigtriangleup T_3)$. This contradicts Fact A  since $T_1\bigtriangleup T_3$ is a balanced $4$-circuit. This proves that $(G,\sigma)$ does not contain $FC_{3}$.

Similarly, in $FC_{4}$, by Fact A, $E_{\phi = 0} = E(T_2)$. Let $C=T_2\bigtriangleup T_3$ which is a balanced $4$-circuit and let $\phi_1 = \phi + \phi(v_4v_5) \chi(C)$.
  Then $\phi_1$ is a $\mathbb Z_3$-flow such that $E_{\phi_1=0} \subseteq \{v_3v_4,v_3v_5\}\subseteq E(T_3\bigtriangleup T_4)$. This contradicts Fact  A since $T_3\bigtriangleup T_4$ is a balanced $4$-circuit. This proves that $(G,\sigma)$ does not contain $FC_{4}$.

 Suppose that $G$ contains $FC_i$ for some $i=5,6,7,8$.  By Fact A,   $\phi(v_4v_5) = 0$  in $FC_5$ and in $FC_i$ where $i=6,7,8$,  $\phi(v_3v_5) = 0$.  Let $C= T_2\bigtriangleup T_3$, which is a balanced $4$-circuit.  Let $\phi_1 = \phi + \phi(v_2v_3) \chi(C)$. Then $\phi_1$ is a $\mathbb Z_3$-flow such that $E_{\phi_1=0} \subseteq E(T_1\bigtriangleup T_2) \cup E(T_4)$ when $i = 5,6$  and $E_{\phi_1=0} \subseteq E(T_1\bigtriangleup T_2) \cup E(T_4\bigtriangleup T_5)$ when $i = 7,8$.  In the former case,  $T_1\bigtriangleup T_2$  is a balanced $4$-circuit and $T_4$ is a balanced $3$-circuit and they are edg-disjoint. In the latter case,  $T_1\bigtriangleup T_2$ and  $T_4\bigtriangleup T_5$ are edge-disjoint balanced $4$-circuits.  This contradicts Fact A and thus proves  that $(G,\sigma)$ does not contain $FC_{i}$ for each $i=5,6,7,8$.

  Now we consider the case when $(G,\sigma)$ contains $FC_9$.  Similar to the above argument, we have  $\phi(v_1v_3) = \phi(v_2v_3)= 0$.

     If $\phi(v_1v_2) = 0$, let  $\phi_1 = \phi + \phi(v_3v_5) \chi(C)$ where $C= T_1\bigtriangleup T_4$ is a balanced $4$-circuit. Then $\phi_1$ is a $\mathbb Z_3$-flow such that $E_{\phi_1=0} \subseteq E(T_3\bigtriangleup T_4)$.  This contradicts Fact  A  since $T_3\bigtriangleup T_4$ is a balanced $4$-circuit.

   Now we further assume $\phi(v_1v_2) = \alpha \not =  0$. Note  that $C= T_1\cup T_3$ is a short barbell. If  one of $\phi(v_3v_4)$ and $\phi(v_3v_5)$ is not equal to $-\alpha$, without loss of generality, assume $\phi(v_3v_4)\not = -\alpha$.  Let $\phi_1 = \phi + \alpha \chi(C)$.  Then $\phi_1$ is a $\mathbb Z_3$-flow such that $E_{\phi_1=0} \subseteq \{v_3v_5,v_5v_4\} \subseteq E(T_2\bigtriangleup T_3)$.  This contradicts Fact A since $T_2\bigtriangleup T_3$ is a balanced $4$-circuit.   If   $\phi(v_3v_4) = \phi(v_3v_5) = -\alpha$, let $\phi_1 =  \phi - \alpha \chi(C)$. Then $\phi_1$ is a $\mathbb Z_3$-flow such that $E_{\phi_1=0} \subseteq \{v_1v_2,v_5v_4\}$, a contradiction to Fact A again since $v_1v_2$ and $v_5v_4$ are contained in the  balanced $4$-circuit $v_1v_2v_4v_5v_1$. This proves that $(G,\sigma)$ does not contain $FC_{9}$.

 Suppose that $(G,\sigma)$ contains $FC_{10}$. Similarly as before we have that   $\phi(v_1v_3) = 0$ and  at least one of  $\phi(v_2v_1) $ and  $\phi(v_2v_3)$ is $0$. Let  $\phi_1 = \phi + \phi(v_4v_5) \chi(C)$ where $C=T_1\cup T_3$ is a short barbell. Then $\phi_1$ is a $\mathbb Z_3$-flow such that $E_{\phi_1=0} \subseteq  E(T_1\bigtriangleup T_4) \cup E(T_2)$. Since $T_1\bigtriangleup T_4$ is a balanced $4$-circuit,   $T_2$ is a balanced triangle, and they share one common edge, by Lemma~\ref{E0:34-circuits}, $(G,\sigma)$ admits a $4$-NZF, a contradiction. Thus $(G,\sigma)$ does not contain $FC_{10}$.

   Finally suppose that $(G,\sigma)$ contains  $FC_{11}$.  Denote $C_1 = T_1\bigtriangleup T_2$,   $C_2 = T_2\bigtriangleup T_3$, and $C_3 = T_4$. Note that  $C_1,C_2,C_3$ are all balanced circuits sharing a common edge $v_2v_4$.

   \begin{claim}
   \label{CL:3-flow}
   There is a $3$-flow $f$ such that $v_2v_4 \in E_{f=0} \subseteq E(C_1)\cup E(C_2)$ and $|E_{f=0} \cap E(C_i)|\geq 2$ for each $i = 1,2$.
   \end{claim}
   \vspace{-0.6cm}
  \begin{proof} With a similar argument as before, we have
    $\phi(v_2v_3) = \phi(v_3v_4) = 0$.
    If $\phi(v_2v_4) =0$,
   then by  Lemma \ref{4-edges:mod3-to-3},   let $f$ be a $3$-flow  with $\supp(f)=\supp(\phi)$ which is a desired $3$-flow.

   Assume $\phi(v_2v_4) = a \not = 0$.  If $\phi(v_1v_2) = \phi(v_1v_3) = b$, let $\phi_1 = \phi + b\chi(C_1)$. Then $E_{\phi_1=0} = \{v_2v_3, v_2v_4\}\subseteq E(C_1)$, a contradiction to Fact A. Thus $\phi(v_1v_2) \not = \phi(v_1v_3)$. Then $a \in \{\phi(v_1v_2), \phi(v_1v_3)\}$.  Let $\phi_2 = \phi - a \chi(C_1)$.  Then $v_2v_4\in E_{\phi_2 = 0}$ and $|E_{\phi_2=0}\cap E(C_i)| = 2$ for each $i =1, 2$. By  Lemma \ref{4-edges:mod3-to-3},   let $g$ be the corresponding $3$-flow of $\phi_2$ with $\supp(g)=\supp(\phi_2)$ which is a desired $3$-flow.
  This prove the claim.
   \end{proof}
 \vspace{-0.3cm}

  Let $f$ be a $3$-flow described in Claim~\ref{CL:3-flow}.
 Note $|\{\pm 1,\pm 2\}\setminus f(C_i)|\ge 2$ for each $i=1,2$.

 If $\{1,-1\}\setminus f(C_i) \not = \emptyset$, take $\alpha_i \in \{1,-1\}\setminus f(C_i)$. Otherwise $f(C_i) = \{0, 1,-1\}$ and  take $\alpha_i \in \{2,-2\}$. In the case when both $|\alpha_1| = |\alpha_2| =2$, we choose $\alpha_1 = 2$ and $\alpha_2=-2$. Then $g = f + \alpha_1 \chi(C_1) + \alpha_2 \chi(C_2)$ is a $4$-flow  such that $E_{g=\pm3} \subseteq E(C_1)\cup E(C_2)$ and $E_{g=0} \subseteq \{v_2v_4\}$.  Since $(G,\sigma)$ does not admit a $4$-NZF, $g(v_2v_4) = 0$.  Since $T_4$ is a balanced triangle and $|g(e)|\leq 2$ for each $e\in E(T_4)$, one can extend $g$ to be a $4$-NZF of $(G,\sigma)$, a contradiction.   This  proves that  $(G,\sigma)$   does not contain $FC_{11}$ and thus completes the proof of  (III).
 \hfill $\Box$

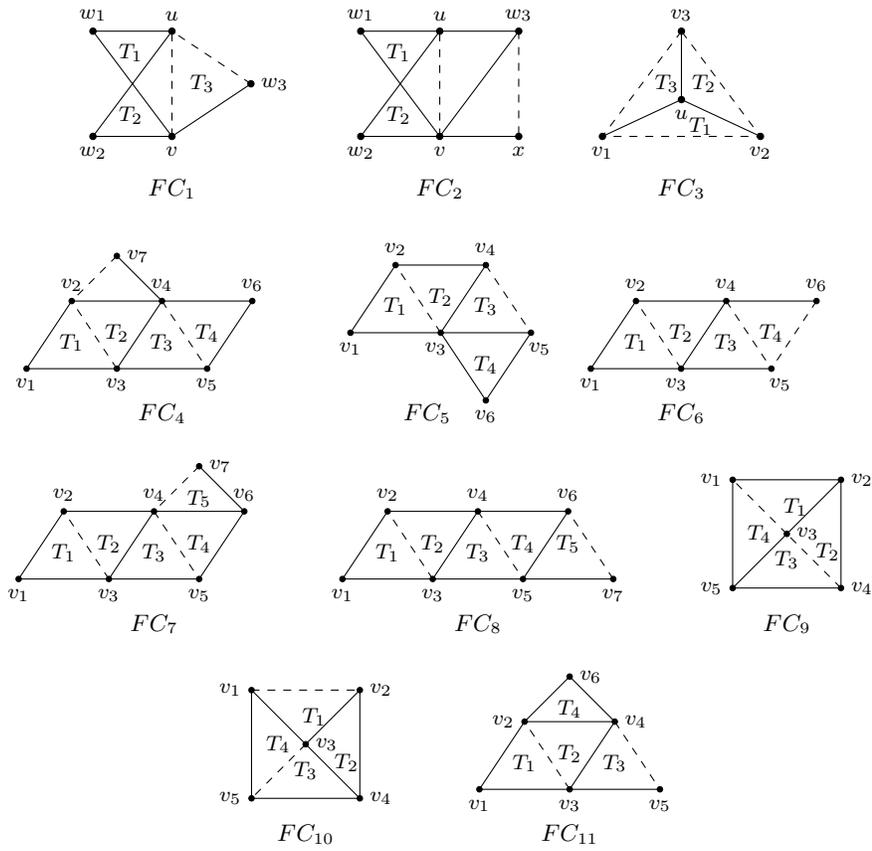
\begin{figure}
\begin{center}
\begin{tikzpicture}[scale=0.7]
%%%%%%%%%%%%%%%%%%%%%%%%%%%%%%%%%%%%%%%%%%%%%%%%%%%%%%%%%%%%%%%%%%%%%%%%%%%%%%%%%%%%%%%%%%%%%%%%%%%%%%%%%%%%%%%%%%%%%%%%%%%%%
\draw [fill=black] (0,2) circle (0.06cm);
\draw [fill=black] (0,0) circle (0.06cm);
\draw [fill=black] (-1.5,2) circle (0.06cm);
\draw [fill=black] (-1.5,0) circle (0.06cm);
\draw [fill=black] (1.5,1) circle (0.06cm);
\draw[dashed] (0,0)--(0,2); \draw (-1.5,2)--(0,2)--(-1.5,0)--(0,0)--(-1.5,2);
\draw (0,0)--(1.5,1);\draw[dashed] (1.5,1)--(0,2);
\node[above] at (0,2) {\scriptsize $u$};\node[below] at (0,0) {\scriptsize $v$};
\node[above] at (-1.5,2) {\scriptsize $w_1$};\node[below] at (-1.5,0) {\scriptsize $w_2$};
\node[right] at (1.5,1) {\scriptsize $w_3$};
\node[right] at (-1.2,1.6) {\scriptsize $T_1$};\node[right] at (-1.2,0.4) {\scriptsize $T_2$};
\node[left] at (1,1) {\scriptsize $T_3$};
\node at (0,-1){ \footnotesize $FC_1$};
\end{tikzpicture}\hspace{0.4cm}
%%%%%%%%%%%%%%%%%%%%%%%%%%%%%%%%%%%%%%%%%%%%%%%%%%%%%%%%%%%%%%%%%%%%%%%%%%%%%%%%%%%
\begin{tikzpicture}[scale=0.7]
\draw [fill=black] (0,2) circle (0.06cm);
\draw [fill=black] (0,0) circle (0.06cm);
\draw [fill=black] (-1.5,2) circle (0.06cm);
\draw [fill=black] (-1.5,0) circle (0.06cm);
\draw [fill=black] (1.5,2) circle (0.06cm);
\draw [fill=black] (1.5,0) circle (0.06cm);
\draw[dashed] (0,0)--(0,2); \draw (-1.5,2)--(0,2)--(-1.5,0)--(0,0)--(-1.5,2);
\draw (0,2)--(1.5,2);\draw[dashed] (1.5,2)--(1.5,0);\draw (1.5,2)--(0,0)--(1.5,0);
\node[above] at (0,2) {\scriptsize $u$};\node[below] at (0,0) {\scriptsize $v$};
\node[above] at (-1.5,2) {\scriptsize $w_1$};\node[below] at (-1.5,0) {\scriptsize $w_2$};
\node[above] at (1.5,2) {\scriptsize $w_3$};\node[below] at (1.5,0) {\scriptsize $x$};
\node[right] at (-1.2,1.6) {\scriptsize $T_1$};\node[right] at (-1.2,0.4) {\scriptsize $T_2$};
%\node[below] at (1.1,1) {\footnotesize $T_4$};
\node at (0,-1){\footnotesize $FC_2$};
\end{tikzpicture}\hspace{0.4cm}
%%%%%%%%%%%%%%%%%%%%%%%%%%%%%%%%%%%%%%%%%%%%%%%%%%%%%%%%%%%%%%%%%%%%%%%%%%%%%%%%%%%%%%%%%%%%%%%%%%%%%%%%%%%%%%%%%%%%%%%%%%%%%
\begin{tikzpicture}[scale=0.7]
\draw [fill=black] (-1.5,0) circle (0.06cm);
\draw [fill=black] (1.5,0) circle (0.06cm);
\draw [fill=black] (0,2) circle (0.06cm);
\draw [fill=black] (0,0.7) circle (0.06cm);
\draw[dashed] (0,2)--(-1.5,0)--(1.5,0)--(0,2);
\draw(0,0.7)--(0,2) (0,0.7)--(-1.5,0) (0,0.7)--(1.5,0);
\node[below] at (-1.5,0) {\scriptsize $v_1$};\node[below] at (1.5,0) {\scriptsize $v_2$};
\node[above] at (0,2) {\scriptsize $v_3$};\node[below] at (0,0.7) {\scriptsize $u$};
\node[right] at (-0.7,1) {\scriptsize $T_3$};\node[right] at (-0.05,0.2) {\scriptsize $T_1$};
\node[left] at (0.85,1) {\scriptsize $T_2$};
\node at (0,-1){\footnotesize $FC_3$};
\end{tikzpicture}
\hspace{0.4cm}

\vspace{0.3cm}

%%%%%%%%%%%%%%%%%%%%%%%%%%%%%%%%%%%%%%%%%%%%%%%%%%%%%%%%%%%%%%%%%%%%%%%%%%%%%%%%%%%%%%%%%%%%%%%%%%%%%%%%%%%%%%%%%%%%%%%%%%%%%
\begin{tikzpicture}[scale=0.6]
\draw [fill=black] (-2,0) circle (0.06cm);
\draw [fill=black] (-1,1.5) circle (0.06cm);
\draw [fill=black] (0,0) circle (0.06cm);
\draw [fill=black] (1,1.5) circle (0.06cm);
\draw [fill=black] (2,0) circle (0.06cm);
\draw [fill=black] (3,1.5) circle (0.06cm);
%\draw [fill=black] (2,2.5) circle (0.06cm);
\draw [fill=black] (0,2.5) circle (0.06cm);
\draw (0,0)--(-2,0)--(-1,1.5)--(1,1.5)--(0,0);\draw[dashed] (0,0)--(-1,1.5);
\draw[dashed] (1,1.5)--(2,0);
\draw (2,0)--(0,0);\draw (3,1.5)--(2,0);\draw(3,1.5)--(1,1.5);
%\draw(2,2.5)--(3,1.5);\draw[dashed](2,2.5)--(1,1.5);
\draw(0,2.5)--(1,1.5);\draw[dashed](0,2.5)--(-1,1.5);
\node[below] at (-2,0) {\scriptsize $v_1$};\node[above] at (-1,1.5) {\scriptsize $v_2$};
\node[below] at (0,0) {\scriptsize $v_3$};\node[above] at (1,1.5) {\scriptsize $v_4$};
\node[below] at (2,0) {\scriptsize $v_5$};\node[above] at (3,1.5) {\scriptsize $v_6$};
\node[right] at (0,2.5) {\scriptsize $v_7$};
\node[below] at (-1,1) {\scriptsize $T_1$};
\node[below] at (0,1.2) {\scriptsize $T_2$};
\node[below] at (1,1) {\scriptsize $T_3$};
\node[below] at (2,1.2) {\scriptsize $T_4$};
%\node[below] at (0,1.2) {\footnotesize $T_2$};
%\node[below] at (0,2.2) {\footnotesize $T_5$};
\node at (1,-1){\footnotesize $FC_{4}$};
\end{tikzpicture}
\hspace{0.5cm}
%%%%%%%%%%%%%%%%%%%%%%%%%%%%%%%%%%%%%%%%%%%%%%%%%%%%%%%%%%%%%%%%%%%%%%%%%%%%%%%%%%%%%%%%%%%%%%%%%%%%%%%%%%%%%%%%%%%%%%%%%%%%%
\begin{tikzpicture}[scale=0.6]
\draw [fill=black] (-2,0) circle (0.06cm);
\draw [fill=black] (-1,1.5) circle (0.06cm);
\draw [fill=black] (0,0) circle (0.06cm);
\draw [fill=black] (1,1.5) circle (0.06cm);
\draw [fill=black] (2,0) circle (0.06cm);
%\draw [fill=black] (3,1.5) circle (0.06cm);
\draw [fill=black] (1,-1.5) circle (0.06cm);
\draw (0,0)--(-2,0)--(-1,1.5)--(1,1.5)--(0,0);\draw[dashed] (0,0)--(-1,1.5);
\draw[dashed] (1,1.5)--(2,0);
\draw (2,0)--(0,0);%\draw[dashed] (3,1.5)--(2,0);\draw(3,1.5)--(1,1.5);
\draw (1,-1.5)--(0,0);\draw(1,-1.5)--(2,0);
\node[below] at (-2,0) {\scriptsize $v_1$};\node[above] at (-1,1.5) {\scriptsize $v_2$};
\node[below] at (-0.1,0) {\scriptsize $v_3$};\node[above] at (1,1.5) {\scriptsize $v_4$};
\node[below] at (2.2,0) {\scriptsize $v_5$};%\node[above] at (3,1.5) {$v_6$};
\node[below] at (1,-1.5) {\scriptsize $v_6$};
\node[below] at (-1,1) {\scriptsize $T_1$};
\node[below] at (0,1.2) {\scriptsize $T_2$};
\node[below] at (1,1) {\scriptsize $T_3$};
\node[below] at (1,-0.3) {\scriptsize $T_4$};
\node at (-0.3,-1.8){ \footnotesize $FC_{5}$};
\end{tikzpicture}
%%%%%%%%%%%%%%%%%%%%%%%%%%%%%%%%%%%%%%%%%%%%%%%%%%%%%%
%%%%%%%%%%%%%%%%%%%%%%%%%%%%%%%%%%%%%%%%%%%%%%%%%%%%%%
\begin{tikzpicture}[scale=0.6]
\draw [fill=black] (-2,0) circle (0.06cm);
\draw [fill=black] (-1,1.5) circle (0.06cm);
\draw [fill=black] (0,0) circle (0.06cm);
\draw [fill=black] (1,1.5) circle (0.06cm);
\draw [fill=black] (2,0) circle (0.06cm);
\draw [fill=black] (3,1.5) circle (0.06cm);
%\draw [fill=black] (1,-1.5) circle (0.06cm);
\draw (0,0)--(-2,0)--(-1,1.5)--(1,1.5)--(0,0);\draw[dashed] (0,0)--(-1,1.5);
\draw[dashed] (1,1.5)--(2,0);
\draw (2,0)--(0,0);\draw[dashed] (3,1.5)--(2,0);\draw(3,1.5)--(1,1.5);
%\draw (1,-1.5)--(0,0);\draw(1,-1.5)--(2,0);
\node[below] at (-2,0) {\scriptsize $v_1$};\node[above] at (-1,1.5) {\scriptsize $v_2$};
\node[below] at (-0.1,0) {\scriptsize $v_3$};\node[above] at (1,1.5) {\scriptsize $v_4$};
\node[below] at (2.2,0) {\scriptsize $v_5$};\node[above] at (3,1.5) {\scriptsize $v_6$};
%\node[below] at (1,-1.5) {$v_7$};
\node[below] at (-1,1) {\scriptsize $T_1$};
\node[below] at (0,1.2) {\scriptsize $T_2$};
\node[below] at (1,1) {\scriptsize $T_3$};
\node[below] at (2,1.2) {\scriptsize $T_4$};
\node at (0,-1){\footnotesize $FC_6$};
\end{tikzpicture}
\hspace{0.5cm}

\vspace{0.2cm}
%%%%%%%%%%%%%%%%%%%%%%%%%%%%%%%%%%%%%%%%%%%%%%%%%%%%%%%%%%%%%%%%%%%%%%%%%%%%%%%%%%%%%%%%%%%%%%%%%%%%%%%%%%%%%%%%%%%%%%%%%%%%%
\begin{tikzpicture}[scale=0.6]
\draw [fill=black] (-2,0) circle (0.06cm);
\draw [fill=black] (-1,1.5) circle (0.06cm);
\draw [fill=black] (0,0) circle (0.06cm);
\draw [fill=black] (1,1.5) circle (0.06cm);
\draw [fill=black] (2,0) circle (0.06cm);
\draw [fill=black] (3,1.5) circle (0.06cm);
\draw [fill=black] (2,2.5) circle (0.06cm);
%\draw [fill=black] (0,2.5) circle (0.06cm);
\draw (0,0)--(-2,0)--(-1,1.5)--(1,1.5)--(0,0);\draw[dashed] (0,0)--(-1,1.5);
\draw[dashed] (1,1.5)--(2,0);
\draw (2,0)--(0,0);\draw (3,1.5)--(2,0);\draw(3,1.5)--(1,1.5);
\draw(2,2.5)--(3,1.5);\draw[dashed](2,2.5)--(1,1.5);
%\draw(0,2.5)--(1,1.5);\draw[dashed](0,2.5)--(-1,1.5);
\node[below] at (-2,0) {\scriptsize $v_1$};\node[above] at (-1,1.5) {\scriptsize $v_2$};
\node[below] at (0,0) {\scriptsize $v_3$};\node[above] at (1,1.5) {\scriptsize $v_4$};
\node[below] at (2,0) {\scriptsize $v_5$};\node[above] at (3,1.5) {\scriptsize $v_6$};
\node[right] at (2,2.5) {\scriptsize $v_7$};
\node[below] at (-1,1) {\scriptsize $T_1$};
\node[below] at (0,1.2) {\scriptsize $T_2$};
\node[below] at (1,1) {\scriptsize $T_3$};
\node[below] at (2,1.2) {\scriptsize $T_4$};
%\node[below] at (1,1) {\footnotesize $T_3$};
\node[below] at (2,2.2) {\scriptsize $T_5$};
\node at (1,-1){ \footnotesize $FC_7$};
\end{tikzpicture}
\hspace{0.5cm}
%%%%%%%%%%%%%%%%%%%%%%%%%%%%%%%%%%%%%%%%%%%%%%%%%%%%%%%%%%%%%%%%%%%%%%%%%%%%%%%%%%%%%%%%%%%%%%%%%%%%%%%%%%%%%%%%%%%%%%%%%%%%%
\begin{tikzpicture}[scale=0.6]
\draw [fill=black] (-2,0) circle (0.06cm);
\draw [fill=black] (-1,1.5) circle (0.06cm);
\draw [fill=black] (0,0) circle (0.06cm);
\draw [fill=black] (1,1.5) circle (0.06cm);
\draw [fill=black] (2,0) circle (0.06cm);
\draw [fill=black] (3,1.5) circle (0.06cm);
\draw [fill=black] (4,0) circle (0.06cm);
\draw (0,0)--(-2,0)--(-1,1.5)--(1,1.5)--(0,0);\draw[dashed] (0,0)--(-1,1.5);
\draw[dashed] (1,1.5)--(2,0);
\draw (2,0)--(0,0);\draw (3,1.5)--(2,0);\draw(3,1.5)--(1,1.5);
\draw(4,0)--(2,0);\draw[dashed](4,0)--(3,1.5);
\node[below] at (-2,0) {\scriptsize $v_1$};\node[above] at (-1,1.5) {\scriptsize $v_2$};
\node[below] at (0,0) {\scriptsize $v_3$};\node[above] at (1,1.5) {\scriptsize $v_4$};
\node[below] at (2,0) {\scriptsize $v_5$};\node[above] at (3,1.5) {\scriptsize $v_6$};
\node[below] at (4,0) {\scriptsize $v_7$};
\node[below] at (-1,1) {\scriptsize $T_1$};
\node[below] at (0,1.2) {\scriptsize $T_2$};
\node[below] at (1,1) {\scriptsize $T_3$};
\node[below] at (2,1.2) {\scriptsize $T_4$};
\node[below] at (3,1.2) {\scriptsize $T_5$};
\node at (1,-1){ \footnotesize $FC_8$};
\end{tikzpicture}
\hspace{0.5cm}
%%%%%%%%%%%%%%%%%%%%%%%%%%%%%%%%%%%%%%%%%%%%%%%%%%%%%%
\begin{tikzpicture}[scale=0.6]
\draw [fill=black] (0,1) circle (0.06cm);
\draw [fill=black] (1.2,2.2) circle (0.06cm);
\draw [fill=black] (-1.2,2.2) circle (0.06cm);
\draw [fill=black] (-1.2,-0.2) circle (0.06cm);
\draw [fill=black] (1.2,-0.2) circle (0.06cm);
\draw[dashed] (0,1)--(-1.2,2.2); \draw (-1.2,2.2)--(0,2.2)--(1.2,2.2)--(0,1);
\draw (0,1)--(-1.2,-0.2)--(1.2,-0.2)--(1.2,2.2);\draw[dashed] (1.2,-0.2)--(0,1);\draw (-1.2,2.2)--(-1.2,-0.2);
\node[right] at (0,1) {\scriptsize  $v_3$};
\node[left] at (-1.2,2.2) {\scriptsize $v_1$};\node[right] at (1.2,2.2) {\scriptsize $v_2$};
\node[right] at (1.2,-0.2) {\scriptsize $v_4$};\node[left] at (-1.2,-0.2) {\scriptsize $v_5$};
\node[right] at (-1.1,1) {\scriptsize $T_4$};
\node[right] at (-0.3,1.6) {\scriptsize $T_1$};
\node[left] at (1.4,0.6) {\scriptsize $T_2$};
\node[above] at (0,0) {\scriptsize $T_3$};
\node at (0,-1){\footnotesize $FC_9$};
\end{tikzpicture}
\vspace{0.2cm}

%%%%%%%%%%%%%%%%%%%%%%%%%%%%%%%%%%%%%%%%%
\begin{tikzpicture}[scale=0.6]
\draw [fill=black] (0,1) circle (0.06cm);
\draw [fill=black] (1.2,2.2) circle (0.06cm);
\draw [fill=black] (-1.2,2.2) circle (0.06cm);
\draw [fill=black] (-1.2,-0.2) circle (0.06cm);
\draw [fill=black] (1.2,-0.2) circle (0.06cm);
\draw (0,1)--(-1.2,2.2) (1.2,2.2)--(0,1); \draw[dashed] (-1.2,2.2)--(1.2,2.2);
\draw[dashed] (0,1)--(-1.2,-0.2);\draw (1.2,-0.2)--(0,1);\draw(-1.2,2.2)--(-1.2,-0.2)--(1.2,-0.2)--(1.2,2.2);
\node[right] at (0,1) {\scriptsize $v_3$};
\node[left] at (-1.2,2.2) {\scriptsize $v_1$};\node[right] at (1.2,2.2) {\scriptsize $v_2$};
\node[right] at (1.2,-0.2) {\scriptsize $v_4$};\node[left] at (-1.2,-0.2) {\scriptsize $v_5$};
\node[right] at (-1.1,1) {\scriptsize $T_4$};
\node[right] at (-0.3,1.6) {\scriptsize $T_1$};
\node[left] at (1.4,0.6) {\scriptsize $T_2$};
\node[above] at (0,0) {\scriptsize $T_3$};
\node at (0,-1){\footnotesize$FC_{10}$};
\end{tikzpicture}
\hspace{0.5cm}
%%%%%%%%%%%%%%%%%%%%%%%%%%%%%%%%%%%%%%%%%%%%%%%%%%%%%%
%%%%%%%%%%%%%%%%%%%%%%%%%%%%%%%%%%%%%%%%%%%%%%%%%%%%%%%%%%%%%%%%%%%%%%%%%%%%%%%%%%%%
%%%%%%%%%%%%%%%%%%%%%%%%%%%%%%%%%%%%%%%%%%%%%%%%%%%%%%%%%%%%%%%%%%%%%%%%%%%%%%%%%%%%%%%%%%%%%%%%%%%%%%%%%%%%%%%%%%%%%%%%%%%%%
\begin{tikzpicture}[scale=0.6]
\draw [fill=black] (-2,0) circle (0.06cm);
\draw [fill=black] (-1,1.5) circle (0.06cm);
\draw [fill=black] (0,0) circle (0.06cm);
\draw [fill=black] (1,1.5) circle (0.06cm);
\draw [fill=black] (2,0) circle (0.06cm);
\draw [fill=black] (0,2.5) circle (0.06cm);
\draw (0,0)--(-2,0)--(-1,1.5)--(1,1.5)--(0,0);\draw[dashed] (0,0)--(-1,1.5);
\draw[dashed] (1,1.5)--(2,0);
\draw (2,0)--(0,0); \draw (0,2.5)--(-1,1.5) (0,2.5)--(1,1.5);
\node[below] at (-2,0) {\scriptsize $v_1$};\node[left] at (-1,1.5) {\scriptsize $v_2$};
\node[below] at (0,0) {\scriptsize $v_3$};\node[right] at (1,1.5) {\scriptsize $v_4$};
\node[below] at (2,0) {\scriptsize $v_5$};\node[right] at (0,2.5) {\scriptsize $v_6$};
\node[below] at (0,2.2) {\scriptsize $T_4$};
\node[below] at (-1,1) {\scriptsize $T_1$};
\node[below] at (0,1.2) {\scriptsize $T_2$};
\node[below] at (1,1) {\scriptsize $T_3$};
\node at (0,-1){\footnotesize $FC_{11}$};
\end{tikzpicture}
\end{center}
\caption{\small Forbidden configurations: the dotted lines are negative edges.}
\label{FIG: confi}
\end{figure}

\medskip \noindent
{\bf (IV)} {\em  There is no triangle-path $T_1T_2\cdots T_m$  in $(G,\sigma)$ such that $m\geq 3$,   $T_1$ and $T_m$ are unbalanced, and  $T_i$ is balanced for each $i \in \{2,\dots, m-1\}$.}

\medskip \noindent
{\bf Proof of (IV).}  Suppose  to the contrary that there is a triangle-path $H= T_1T_2\cdots T_m$  such that $m\geq 3$,  $T_1$ and $T_m$ are unbalanced and $T_i$ is balanced for each $i = \{2,\dots, m-1\}$.  Denote by $H' = T_2\cdots T_{m-1}$.
Denote $E(T_1) = \{e_1,e_2,e_3\}$ and $E(T_m) = \{e_4,e_5,e_6\}$ where $e_3 \in E(T_1)\cap E(T_2)$ and $e_6 \in E(T_m)\cap E(T_{m-1})$. Let $x$  be the common endvertex of $e_1$ and $e_2$ and $y$ be the common endvertex of $e_4$ and $e_5$.  Let $C=T_1\bigtriangleup T_2\bigtriangleup \cdots \bigtriangleup T_m$. Then $C$ is a balanced circuit containing $e_i$ for each $i =1,2,4,5$.

Take  $T = T_1$. Then $E_{\phi=0} \subseteq E(T_1)$. Since $e_3$ belongs to the  balanced triangle $T_2$, by Lemma \ref{E0:2-circuits}, either $\phi(e_1) = 0$ or $\phi(e_2) = 0$.

 If $d(x) \geq 3$,  there is a triangle $T_0$ such that $T_0$ and $T_1$ share exactly one  of $e_1$ and $e_2$ since by (II), $G$ is locally connected.  Let $C_1 = T_0$ if $T_0$ is balanced otherwise let $C_1 = T_0\bigtriangleup T_1$ which is a balanced $4$-circuit.  Without loss of generality assume $e_1 \in E(C_1)$.

 Similarly if $d(y) \geq 3$, there is a triangle $T_{m+1}$ such that $T_{m+1}$ and $T_m$ share exactly one  of $e_4$ and $e_5$.  Let $C_2 = T_{m+1}$ if $T_{m+1}$ is balanced otherwise let $C_2 = T_{m+1}\bigtriangleup T_m$ which is a balanced $4$-circuit.  Without loss of generality assume $e_4 \in E(C_2)$.

 Let $\alpha = \phi(e_5)$ and  $\phi_1 = \phi + \alpha \chi(C)$.

 We first show $\phi(e_1) \not = \phi(e_2)$.  Suppose the contradiction that  $\phi(e_1) = \phi(e_2)$. Then  $\phi(e_1) = \phi(e_2) = 0$ and thus $E_{\phi_1=0} \subseteq E(H') \cup \{e_4\}$.

 If $\phi_1(e_4) \not = 0$, then $E_{\phi_1=0} \subseteq E(H')$. By  Lemma~\ref{shift-0}, there is a $Z_3$-flow $\phi_2$ such that $E_{\phi_2=0} \subseteq  \{e_6\}$, a contradiction to Fact  A.

  If $\phi_1(e_4) = 0$, then $\phi(e_4) \not = \phi(e_5)$. This implies $d(y) \geq 3$ and thus $C_2$ exists. If $E(C_2) \cap E(H') \not = \emptyset$, let $e_0\in E(C_2) \cap E(H')$. Otherwise, let $e_0 = e_6$.   By Lemma~\ref{shift-0}, there is a $Z_3$-flow $\phi_3$ such that $E_{\phi_3=0} \subseteq  \{e_0, e_4\} \subseteq E(C_2)$, a contradiction to Fact A since  $C_2$ is a balanced circuit of length at most $4$. This shows that  $\phi(e_1) \not = \phi(e_2)$, which implies $d(x) \geq 3$. By symmetry, we also have $d(y) \geq 3$. Therefore both $C_1$ and $C_2$ exist.

 Since $e_1\in E(C_1)$ and $e_3 \in E(T_2)$, we have $\phi(e_2) = 0$. Then $E_{\phi_1= 0} \subseteq E(H') \cup \{e_1,e_4\}$.  If $(E(C_1)\cup E(C_2))\cap E(H') \not = \emptyset$, let $e_7$ be an edge in $(E(C_1)\cup E(C_2))\cap E(H')$. Otherwise let $e_7 = e_3$. By Lemma~\ref{shift-0},  one can obtained a  $Z_3$-flow $\phi_4$ from $\phi_1$ such that $E_{\phi_4=0} \subseteq  \{e_1, e_4, e_7\}$. Note that if $C_i$ is a circuit of length  $4$ for some $i = 1,2$, then $e_7  \in E(C_i)\cap E(H')$.

If $C_1$ and $C_2$ are edge-disjoint, then we have either $ \{e_1, e_4, e_7\} \subseteq E(C_1)\cup E(C_2)$ or
$ \{e_1, e_4, e_7\} \subseteq E(C_1)\cup E(C_2) \cup E(T_2)$ where $C_1,C_2, T_2$ are  edge-disjoint balanced triangles. The former case contradicts Fact A. In the latter case, by Lemma~\ref{4-edges:mod3-to-3}, there is an integer $3$-flow flow $f$ such that $\supp(f) = \supp(\phi_4)$. By Lemma~\ref{f-extension} (considering $f$ as an integer $4$-flow),  $f$ can be extended to a $4$-NZF of  $G$, a contradiction.  Therefore $C_1$ and $C_2$ are not edge-disjoint.

If  $C_1$ is a triangle, then by  Lemma~\ref{shift-0},  one can obtain a  $Z_3$-flow $\phi_5$ from $\phi_4$ such that   $|E_{\phi_5=0}|\leq 4$ and  $E_{\phi_5=0} \subseteq  E(C_2)\cup  \{e_7\}$ since $C_1$ and $C_2$ are not  edge-disjoint.  Since $e_7$ is contained in a balanced triangle and $C_2$ is a balanced $4$-circuit, by Lemma \ref{E0:2-circuits} or Lemma~\ref{E0:34-circuits}, $(G,\sigma)$ has a $4$-NZF, a contradiction. Thus $C_1$ is a $4$-circuit. By symmetry, $C_2$ is also a $4$-circuit. This implies $e_3 \in E(C_1)$ and $e_6\in E(C_2)$ and $\{e_1, e_4\} \subseteq E_{\phi_4 = 0 } \subseteq \{e_1,e_4, e_7\}\subseteq  E(C_1)\cup E(C_2)$.

Since $C_1$ and $C_2$ are not edge-disjoint,   there is a $\beta \in \mathbb Z_3$ such that $\phi_6 = \phi_4 + \beta \chi(C_1)$ satisfying $E_{\phi_6=0} \subseteq E(C_2) \cup \{e_3\}$. Since $\{e_3, e_7 \} \subseteq E(H')$, by   Lemma~\ref{shift-0},  one can obtain a  $Z_3$-flow $\phi_7$ from $\phi_6$ such that   $E_{\phi_7=0} \subseteq E(C_2)$, a contradiction to Fact A. This completes the proof of (IV).

\medskip \noindent
{\bf (V)} For any triangle-path $H = T_1T_2T_3$  with each $T_i$ unbalanced, $H$ is an induced subgraph of $(G,\sigma)$.

\medskip \noindent
{\bf Proof of (V).} Suppose to the contrary that $H$ is  not induced.  Denote $V(H) = \{v_1,v_2,v_3,v_4,v_5\}$ where $V(T_i) = \{v_i, v_{i+1}, v_{i+2}\}$ for each $i = 1,2,3$.

Since by (III), $(G,\sigma)$ does not contain $FC_9$ or $FC_{10}$, $v_1$ and $v_5$ are not adjacent.  Then either $v_1$ and $v_4$ are adjacent or $v_2$ and $v_5$ are adjacent. Without loss of generality, assume $v_1$ and $v_4$ are adjacent. Denote $T_4 = v_1v_3v_4$. Since by (III) $(G,\sigma)$ does not contain $FC_3$, $T_4$ is balanced.  Then $T_2,T_3$ and $T_4$  form a $FC_1$, a contradiction to (III) again. This completes the proof of (V).

\begin{figure}
\begin{center}
\begin{tikzpicture}[scale=0.6]
\draw [fill=black] (-2,0) circle (0.06cm);
\draw [fill=black] (-1,1.5) circle (0.06cm);
\draw [fill=black] (0,0) circle (0.06cm);
\draw [fill=black] (1,1.5) circle (0.06cm);
\draw [fill=black] (2,0) circle (0.06cm);
\draw [fill=black] (3,1.5) circle (0.06cm);
\draw (0,0)--(-2,0)--(-1,1.5)--(1,1.5)--(0,0);\draw[dashed] (0,0)--(-1,1.5);
\draw[dashed] (1,1.5)--(2,0);
\draw (2,0)--(0,0);\draw (3,1.5)--(2,0);\draw(3,1.5)--(1,1.5);
\node at (0,-0.8){$\Gamma_1$};
\end{tikzpicture}
\begin{tikzpicture}[scale=0.6]
\draw [fill=black] (-2,0) circle (0.06cm);
\draw [fill=black] (-1,1.5) circle (0.06cm);
\draw [fill=black] (0,0) circle (0.06cm);
\draw [fill=black] (1,1.5) circle (0.06cm);
\draw [fill=black] (2,0) circle (0.06cm);
\draw [fill=black] (0,2.5) circle (0.06cm);
\draw (0,0)--(-2,0)--(-1,1.5)--(1,1.5)--(0,0);\draw[dashed] (0,0)--(-1,1.5);
\draw[dashed] (1,1.5)--(2,0);
\draw (2,0)--(0,0); \draw[dashed] (0,2.5)--(-1,1.5); \draw(0,2.5)--(1,1.5);
\node at (0,-0.8){$\Gamma_2$};
\end{tikzpicture}
\begin{tikzpicture}[scale=0.6]
\draw [fill=black] (0,0) circle (0.06);
\draw [fill=black] (18:1.5) circle (0.06);
\draw [fill=black] (90:1.5) circle (0.06);
\draw [fill=black] (162:1.5) circle (0.06);
\draw [fill=black] (234:1.5) circle (0.06);
\draw [fill=black] (306:1.5) circle (0.06);
\draw (18:1.5)--(90:1.5)--(162:1.5)--(234:1.5);
\draw (306:1.5)--(18:1.5);
\draw  (0,0)--(90:1.5)  (0,0)--(234:1.5) (0,0)--(306:1.5);
\draw[dashed](0,0)--(162:1.5) (0,0)--(18:1.5);
\node at (0,-2){$\Gamma_3$};
\end{tikzpicture}
\end{center}
\caption{\small Three graphs formed by four unbalanced triangles}
\label{FIG: confi-2}
\end{figure}
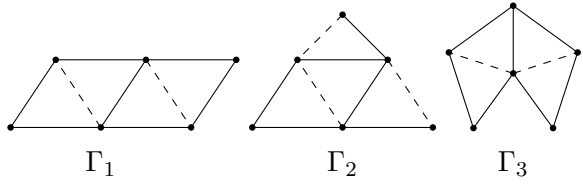

\medskip  \noindent
{\bf The final step.}
  By (III), $(G,\sigma)$ does not contain any graph of Figure \ref{FIG: confi} as a subgraph. We can further assume that $(G,\sigma)$ contains two edge-disjoint unbalanced triangles by (I).

By (IV), let  $H=T_1T_2\ldots T_m$ be a triangle-path such that each triangle $T_i$ is  unbalanced and $E(T_1)\cap E(T_m)=\emptyset$. We choose $H$ such that $m$ is as large as possible.  Since  $(G,\sigma)$ contains two edge-disjoint unbalanced triangles by (I) and   does not contain $FC_8$ by (III), we have $3\leq m \leq 4$.  One can easily see that $H$ admits a $4$-NZF.  Since $(G,\sigma)$ does not admit a $4$-NZF, $H \not = G$. Since $G$ is triangularly connected, there must be a triangle $T_5 \not =T_i$ for each $i=1,2,3$ such that $E(T_5)\cap E(H) \not = \emptyset$.

 If $m=4$, then $H=\Gamma_1$ or $\Gamma_3$ in Figure~\ref{FIG: confi-2}. If $m=3$,  by (V), $H$ is an induced subgraph and hence $|E(T_4)\cap E(H)|=1$.  Since by (III), $G$ does not contain $FC_i$ for each $i = 1, 2,5,6, 11$, $H$ must be  one of  $\Gamma_i$  in Figure~\ref{FIG: confi-2}.  It is easy to see that each $\Gamma_i$ admits a $4$-NZF and thus $(G,\sigma) \not = \Gamma_i$ for each $i$. Since $G$ is triangularly connected, there is a triangle $T_6$ such that $T_6 \not = T_i$ for each $i =1,2,3,4$ and $E(T_6)\cap E(H) \not = \emptyset$.  By the maximality of $m$ and since $(G,\sigma)$ does not contain $FC_i$ for each $i =1,2,4,5,6,11$,  we have $|E(T_6)\cap E(H)| \geq 2$. By (V), $H= \Gamma_3$ and  thus by (IV) $G=(W_5,\sigma^*)$, a contradiction. This completes the proof of the theorem.
\hfill $\Box$

%\newpage

\end{document}